\documentclass[12pt]{article}

\usepackage{amsmath}
\usepackage{amsmath,amsfonts,amssymb,amsthm,graphics,amscd}
\usepackage[lmargin=3cm, rmargin=3cm,tmargin=2.5cm,bmargin=2.5cm]{geometry}

\allowdisplaybreaks[4]
\newtheorem{theorem}{Theorem}
\theoremstyle{plain}

\newtheorem{lemma}{Lemma}
\newtheorem{Hypothesis}{Hypothesis}
\newtheorem{notation}{Notation}
\newtheorem{Inverse problem}{Inverse Problem}

\newtheorem{remark}{Remark}

\numberwithin{equation}{section}
\begin{document}
\title
{Spectral properties of
Sturm-Liouville operators on infinite metric graphs}

\author{Yihan Liu, Jun Yan* and Jia Zhao }
\maketitle

\setcounter{page}{1}

\begin{abstract}
This paper mainly deals with the Sturm-Liouville operator
\begin{equation*}
\mathbf{H}=\frac{1}{w(x)}\left( -\frac{\mathrm{d}}{\mathrm{d}x}p(x)\frac{
\mathrm{d}}{\mathrm{d}x}+q(x)\right) ,\text{ }x\in \Gamma
\end{equation*}
acting in $L_{w}^{2}\left( \Gamma \right) ,$ where $\Gamma $ is a metric
graph. We establish a relationship between the bottom of the spectrum and
the positive solutions of quantum graphs, which is a generalization of the
classical Allegretto-Piepenbrink theorem. Moreover, we prove the
Persson-type theorem, which characterizes the infimum of the essential
spectrum.
\end{abstract}

\renewcommand{\thefootnote}{}
\footnotetext{\hspace{-7pt} {\em 2020 Mathematics Subject
Classification\/}:Primary 34B45; Secondary 34L05, 81Q35
\baselineskip=18pt\newline\indent {\em Key words and phrases\/}:
Sturm-Liouville operators, metric graphs, spectrum.}

\section{Introduction}

The main object of the present paper is the self-adjoint Sturm-Liouville
operator in the Hilbert space $L_{w}^{2}(\Gamma )$ associated with the
differential expression
\begin{equation}
lf(x)=\frac{1}{w(x)}\left( -\left( p(x)f^{\prime }(x)\right) ^{\prime
}+q(x)f(x)\right) ,\text{ }x\in \Gamma ,  \label{Schrodinger}
\end{equation}%
where $\Gamma $ is a metric graph and the matching conditions imposed at
inner vertices are the \emph{Kirchhoff conditions}. Throughout this paper,
we always assume that $1/p,$ $q,$ $w\in L_{loc}^{1}\left( \Gamma \right) .$

In the last two decades, differential operators on metric graphs have
attracted huge attentions due to numerous applications in mathematical
physics and engineering (\cite{6,8,9,10,kuch} and references therein).
Particularly, there has been an increasing interest in the spectral theory
of Sturm-Liouville operators on metric graphs. (see \cite{spec,8,9,kuch,S1}
and references therein). From the mathematical point of view, such a system
is interesting because it exhibits a mixed dimensionality being locally
one-dimensional but globally multi-dimensional of many different types.

Consider the following form in $L_{w}^{2}(\Gamma )$
\begin{equation}
\mathbf{t}_{0}^{0}\left[ f\right] =\int\nolimits_{\Gamma }p(x)\left\vert
f^{\prime }(x)\right\vert ^{2}\mathrm{d}x,\text{ }\mathbf{q}\left[ f\right]
=\int\nolimits_{\Gamma }q\left\vert f(x)\right\vert ^{2}\mathrm{d}x
\end{equation}%
defined on the respective domains
\begin{equation*}
\text{dom}(\mathbf{t}_{0}^{0})=\{f\in H_{c}^{1}(\Gamma ;p,w):f|_{\partial
\Gamma }=0\}
\end{equation*}%
and%
\begin{equation*}
\text{dom}(\mathbf{q})=\{f\in L_{w}^{2}(\Gamma ):\left\vert \mathbf{q}\left[
f\right] \right\vert <\infty \}.
\end{equation*}%
Here $H_{c}^{1}(\Gamma ;p,w)$ denotes the subspace of $H^{1}(\Gamma ;p,w)$
with compact support, and%
\begin{equation*}
\begin{array}{ll}
H^{1}(\Gamma ;p,w):= & \{f\in L_{w}^{2}(\Gamma ):\int\nolimits_{\Gamma
}p(x)\left\vert f^{\prime }(x)\right\vert ^{2}\mathrm{d}x<\infty , \\
& f\text{ is continuous and edgewise absolutely continuous}\}.%
\end{array}%
\end{equation*}%
Let us introduce the form $\mathbf{t}_{q}^{0}$ as a form sum of the two
forms $\mathbf{t}_{0}^{0}$ and $\mathbf{q:}$%
\begin{equation*}
\mathbf{t}_{q}^{0}\left[ f\right] =\mathbf{t}_{0}^{0}\left[ f\right] +%
\mathbf{q}\left[ f\right] ,\text{ dom}\left( \mathbf{t}_{q}^{0}\right) =%
\text{dom}(\mathbf{t}_{0}^{0})\cap \text{dom}(\mathbf{q}).
\end{equation*}%
If the form $\mathbf{t}_{q}^{0}$ is lower semibounded, then it is closable.
Denote $\mathbf{H}_{\mathbf{t}_{q}}$ the self-adjoint operator associated
with $\mathbf{t}_{q}=\overline{\mathbf{t}_{q}^{0}}.$

The purpose of this paper is to develop Allegretto-Piepenbrink-type theorem
(Theorem \ref{AP}) and Persson-type theorem (Theorem \ref{persson}) for the
operator $\mathbf{H}_{\mathbf{t}_{q}}$, which are classical topics and we
refer to the papers cited in this paragraph for historical remarks (\cite%
{11,12,13,22,singu,deco,akdu}). More precisely, we establish a
relationship between the bottom of the spectrum and the positive solutions
of quantum graphs, which is a generalization of the classical
Allegretto-Piepenbrink theorem. Moreover, we prove the Persson-type theorem,
which characterizes the infimum of the essential spectrum. It should be
mentioned that the quantities $\inf \sigma \left( \mathbf{H}_{\mathbf{t}%
_{q}}\right) $ and $\inf \sigma _{ess}\left( \mathbf{H}_{\mathbf{t}%
_{q}}\right) $ are of fundamental importance for several reasons. For
example, in the theory of parabolic equations the quantity $\inf \sigma
\left( \mathbf{H}_{\mathbf{t}_{q}}\right) $ may give the speed of
convergence of the system towards equilibrium. Moreover, $\inf \sigma
_{ess}\left( \mathbf{H}_{\text{D}}\right) =+\infty $ holds if and only if $%
\mathbf{H}_{\mathbf{t}_{q}}$ has purely discrete spectrum.

In the last decades, Allegretto-Piepenbrink-type theorem has been
investigated for strongly local Dirichlet forms \cite{22}, positive Schr\"{o}%
dinger operators on general weighted graphs \cite{critical graph} and Schr%
\"{o}dinger operators on $%
\mathbb{R}
^{d}$ with singular potentials \cite{singu}. Here, we generalize this
theorem to our context and provide a simple proof along the lines of $\left(
\cite{hartman,positive}\right) $.

Recently, a Persson-type theorem for the Schr\"{o}dinger operators ($p=w=1$)
on infinite metric graphs has been given in \cite{akdu} by Akduman and
Pankov. Moreover, Lenz and Stollmann present a Persson-type theorem valid
for all regular Dirichlet forms satisfying a spatial local compactness
condition \cite{deco}; they also discuss a generalization to certain Schr\"{o%
}dinger type operators, where the negative part of measure perturbations has
to fulfill some Kato condition. In this direction, we present concrete
conditions on the coefficients $1/p,$ $q,$ $w$ and the lengths of the graph
edges, which guarantee the validity of the Persson-type theorem for the
Sturm-Liouville operator on infinite metric graphs.

Let us now finish the introduction by describing the content of the article.
In Sect. 2, we review necessary notions and facts on infinite metric graphs.
Section 3 and Section 4 are devoted to investigating the
Allegretto-Piepenbrink-type theorem and the Persson-type theorem for the
Sturm-Liouville operator $\mathbf{H}_{\mathbf{t}_{q}}$ on infinite metric
graphs.

\section{Preliminaries on metric graphs}

In what follows, $\Gamma =(\mathcal{E},\mathcal{V})$ will be an undirected
graph with countably infinite sets of vertices $\mathcal{V}$ and edges $%
\mathcal{E}$. A graph is called connected if for any two vertices there is a
path connecting them. For every vertex $v\in \mathcal{V}$, we denote the set
of edges incident to the vertex $v$ by $\mathcal{E}_{v}\ $and%
\begin{equation*}
\deg _{\Gamma }\left( v\right) :=\#\left\{ e:e\in \mathcal{E}_{v}\right\}
\end{equation*}%
is called \textit{the} \textit{degree} of a vertex $v\in \mathcal{V}$.
Moreover, the boundary of $\Gamma $ is defined as
\begin{equation*}
\partial \Gamma =\left\{ v\in \mathcal{V}\left( \Gamma \right) :\deg
_{\Gamma }\left( v\right) =1\right\} .
\end{equation*}%
The graph $\Gamma $ is said to be a \textit{metric graph} if each edge $e$
is assigned a positive length $|e|\in (0,\infty )$. This enables us to equip
$\Gamma $ with a topology and metric. By assigning each edge a direction and
calling one of its vertices the initial vertex $o\left( e\right) $ and the
other one the terminal vertex $t\left( e\right) $, every edge $e\in \mathcal{%
E}\left( \Gamma \right) $ can be identified with a copy of the interval $%
I_{e}=[0,|e|]$. The distance $\rho (x,y)$ between two points $x$ and $y$ in $%
\Gamma $ is defined as the length of the shortest path that connects these
points. Since the graph is connected, the distance is well defined. In
addition, there is a natural measure, d$x$, on $\Gamma $ which coincides
with the Lebesgue measure on each edge. In particular, integration over $%
\Gamma $ makes sense. For further details we refer to, e.g., \cite[Chapter
1.3]{kuch}.

Throughout this paper, we shall always make the following assumptions.

\begin{Hypothesis}
The graph $\Gamma $ is connected and locally finite $(\deg _{\Gamma }\left(
v\right) <\infty $ for every $v\in \mathcal{V})$.
\end{Hypothesis}

\begin{Hypothesis}
\emph{(Finite ball condition)}. For any positive number $r$ and any vertex $%
v $ there is only a finite number of vertices $w$ at a distance smaller than
$r $ from $v$.
\end{Hypothesis}

\begin{Hypothesis}
\label{shangjie}There is a finite upper bound for lengths of graph edges:%
\begin{equation*}
\sup\limits_{e\in \mathcal{E}\left( \Gamma \right) }\left\vert e\right\vert
=d^{\ast }<\infty .
\end{equation*}
\end{Hypothesis}

In fact, Hypothesis \ref{shangjie} is imposed for a convenience only. We
denote by $L_{w}^{2}(\Gamma )$ the space of all complex-valued functions
which are weighted square-integrable on $\Gamma $ with respect to the
measure d$x$. More explicitly, this space consists of all measurable
functions $f$ such that $f_{e}\in $ $L_{w}^{2}(e)$ for all $e\in \mathcal{E}%
(\Gamma )$ and
\begin{equation*}
\left\Vert f\right\Vert _{L_{w}^{2}(\Gamma )}^{2}=\sum_{e\in \mathcal{E}%
(\Gamma )}\left\Vert f\right\Vert _{L_{w}^{2}(e)}^{2}<\infty .
\end{equation*}%
We also need the standard space $L_{loc}^{1}\left( \Gamma \right) $ with
respect to the measure d$x$. It consists of all functions which are
absolutely integrable on every edge.

In this paper, we impose at each inner vertex the following conditions:%
\begin{equation}
\left\{
\begin{array}{c}
f\text{ is continuous at }v, \\
\sum\limits_{e\in \mathcal{E}_{v}}\left( pf^{\prime }\right) _{e}\left(
v\right) =0,%
\end{array}%
\right.  \label{kirchhoff}
\end{equation}%
which is the so-called Kirchhoff vertex conditions. Here $\left( pf^{\prime
}\right) _{e}\left( v\right) $ is the quasi-derivative in the outgoing
direction at the vertex $v$, $f_{e}$ denotes the restriction of a function $%
f $ onto the edge $e$.

\begin{notation}
In this paper, $%
\mathbb{N}
$ denotes the set of positive integers and $%
\mathbb{N}
_{0}$ denotes the set of nonnegative integers.
\end{notation}

\section{Allegretto-Piepenbrink-type theorem on metric graphs}

In this section, we establish the Allegretto-Piepenbrink-type theorem for
the operator $\mathbf{H}_{\mathbf{t}_{q}}$. It should be mentioned that we
allow multigraphs, that is, we allow multiple edges and loops. We shall say
that a function is a solution of the equation $ly=\lambda y,$ $\lambda \in
\mathbb{C}
$ if it satisfies the equation on each edge $e\in \mathcal{E}\left( \Gamma
\right) $ and satisfies the Kirchhoff conditions at inner vertices.

\begin{theorem}[Allegretto-Piepenbrink-type theorem]
\label{AP}For any $\lambda \in \mathbb{R}$,\newline
$(1)$ if there exists a positive solution $y>0$ on $\Gamma $ for $ly=\lambda
y$, then $\inf \sigma (\mathbf{H}_{\mathbf{t}_{q}})\geq \lambda ;$

$\hspace{-4mm}(2)$ if $\inf \sigma (\mathbf{H}_{\mathbf{t}_{q}})>\lambda $,
then there exists a positive solution $y>0$ on $\Gamma $ for $ly=\lambda y$.
\end{theorem}

\begin{proof}
(1) Let $y$ be a positive solution of $ly=\lambda y.$ Then $\forall \eta \in
$\textrm{{\ $\mathrm{dom}$}}$\mathrm{(\mathbf{t}_{q}^{0})}$, denote $g(x)=%
\frac{\eta (x)}{y(x)}$, and thus we have $g\in $\textrm{{\ $\mathrm{dom}$}}$%
\mathrm{(\mathbf{t}_{q}^{0})}$. Note that
\begin{eqnarray*}
\mathbf{t}_{q}^{0}\left[ \eta \right] &=&\int\nolimits_{\Gamma
}p(x)\left\vert \eta ^{\prime }(x)\right\vert ^{2}+q(x)\left\vert \eta
(x)\right\vert ^{2}\mathrm{d}x, \\
&=&\int_{\Gamma }[p|g^{\prime }y|^{2}+p|gy^{\prime }|^{2}+pg^{\prime }y%
\overline{g}\overline{y}^{\prime }+p\overline{g}^{\prime }\overline{y}%
gy^{\prime }]+q(x)\left\vert \eta (x)\right\vert ^{2}\mathrm{d}x,
\end{eqnarray*}%
and%
\begin{eqnarray*}
\int_{\Gamma }p|g|^{2}|y^{\prime }|^{2}\mathrm{d}x &=&\left.
p|g|^{2}y^{\prime }\overline{y}\right\vert _{\partial \Gamma }-\int_{\Gamma
}(p|g|^{2}y^{\prime })^{\prime }\overline{y}\mathrm{d}x \\
&=&-\int_{\Gamma }(py^{\prime })^{\prime }|g|^{2}\overline{y}\mathrm{d}%
x-\int_{\Gamma }py^{\prime }\overline{y}(g^{\prime }\overline{g}+\overline{g}%
^{\prime }g)\mathrm{d}x.
\end{eqnarray*}%
Therefore, $\forall \eta \in $\textrm{{\ $\mathrm{dom}$}}$\mathrm{(\mathbf{t}%
_{q}^{0})}$,%
\begin{equation*}
\mathbf{t}_{q}^{0}\left[ \eta \right] =\int_{\Gamma }\left( p|g^{\prime
}y|^{2}+\lambda w|g|^{2}|y|^{2}\right) \mathrm{d}x\geq \lambda \int_{\Gamma
}w|\eta |^{2}\mathrm{d}x.
\end{equation*}%
This shows that the lower bound of the form $\mathbf{t}_{q}$ and thus the
operator $\mathbf{H}_{\mathbf{t}_{q}}$ is not less than $\lambda .$ \newline
(2) Assume that $\inf \sigma (\mathbf{H}_{\mathbf{t}_{q}})>\lambda $. Let $%
\Gamma ^{\prime }$ $\subset \Gamma $ be any finite compact subgraph obtained
by cutting through the interior of edges$.$ Denote
\begin{equation}
\mathbf{t}_{q,\Gamma ^{\prime }}^{0}\left[ f\right] :=\int\nolimits_{\Gamma
^{\prime }}p(x)\left\vert f^{\prime }(x)\right\vert ^{2}+q(x)\left\vert
f(x)\right\vert ^{2}\mathrm{d}x,
\end{equation}%
and%
\begin{equation*}
\text{dom}(\mathbf{t}_{q,\Gamma ^{\prime }}^{0}):=\{f\in H^{1}(\Gamma
^{\prime };p,w):f|_{\partial \Gamma ^{\prime }}=0,\text{ }\mathbf{t}%
_{q,\Gamma ^{\prime }}^{0}\left[ f\right] <\infty \}.
\end{equation*}%
Note that dom$(\mathbf{t}_{q,\Gamma ^{\prime }}^{0})\subset $ dom$(\mathbf{t}%
_{q}^{0})$ in the sense that every function in dom$(\mathbf{t}_{q,\Gamma
^{\prime }}^{0})$ can be extended to be in dom$(\mathbf{t}_{q}^{0})$ by
setting it zero on remaining edges. Thus the form $\mathbf{t}_{q,\Gamma
^{\prime }}^{0}$ is lower semibounded and then is closable. According to the
representation theorem, it is standard to show that the self-adjoint
operator associated with the closure $\mathbf{t}_{q,\Gamma ^{\prime }}=%
\overline{\mathbf{t}_{q,\Gamma ^{\prime }}^{0}}$ is the Dirichlet operator $%
\mathbf{H}_{\text{D}}^{\Gamma ^{\prime }}$ in $L_{w}^{2}(\Gamma ^{\prime }),$
which directly yields that%
\begin{equation}
\inf \sigma \left( \mathbf{H}_{\text{D}}^{\Gamma ^{\prime }}\right) >\lambda
.  \label{subda}
\end{equation}

Fix a vertex $o\in \Gamma .$ For any integer $n>0$, denote%
\begin{equation*}
\Gamma \left( o;nd^{\ast }\right) :=\left\{ x\in \Gamma \left\vert \rho
\left( o,x\right) \leq nd^{\ast }\right. \right\} .
\end{equation*}%
Then cutting through the interior of edges, we can create a subgraph $\Gamma
_{n}\ $such that $\Gamma \left( o;nd^{\ast }\right) \subset \Gamma
_{n}\subset \Gamma \left( o;\left( n+1\right) d^{\ast }\right) $. In view of
\cite[Theorem 4.1 and Theorem 5.2.2]{POKO} and (\ref{subda}), it is easy to
see that there exists a solution $u_{n}\left( x\right) $ of $ly=\lambda y$
that is positive on $\Gamma _{n}.$ Define $y_{n}\left( x\right) =\frac{%
u_{n}\left( x\right) }{u_{n}\left( o\right) }.$ Then it can be seen that
\begin{equation*}
ly_{n}=\lambda y_{n}\text{, }\inf_{x\in \Gamma _{n}}y_{n}\left( x\right) >0%
\text{ and }y_{n}\left( o\right) =1.
\end{equation*}

Next, we shall prove that $\left\{ y_{n}\right\} _{n=m}^{\infty }$ and $%
\left\{ py_{n}^{\prime }\right\} _{n=m}^{\infty }\ $are uniformly bounded
and equicontinuous on each $e\in \mathcal{E}\left( \Gamma _{m}\right) .$ By
using the Harnack inequality given by Gesteszy \cite{ges} and the Kirchhoff
conditions at inner vertices, it is easy to obtain that for each positive
integer $n,$ there exist\ positive constants $C_{m1}$ and $C_{m2}$
(depending only on $m$) such that%
\begin{equation}
C_{m2}\leq \inf_{_{\substack{ x\in \Gamma _{m}  \\ m<n}}}y_{n}\left(
x\right) \leq \sup_{\substack{ x\in \Gamma _{m}  \\ m<n}}y_{n}\left(
x\right) \leq C_{m1}.  \label{harnack}
\end{equation}%
This means that $\left\{ y_{n}\left( x\right) \right\} _{n=m}^{\infty }$ is
uniformly bounded on $\Gamma _{m}$. For each $e\in \mathcal{E}\left( \Gamma
_{m}\right) $ and $x_{1},x_{2}\in e,$%
\begin{equation*}
\left\vert y_{n}\left( x_{1}\right) -y_{n}\left( x_{2}\right) \right\vert
=\int_{x_{1}}^{x_{2}}\frac{1}{p}py_{n}^{\prime }\mathrm{d}t\leq \sup_{x\in
e}\left\vert py_{n}^{\prime }\right\vert \int_{x_{1}}^{x_{2}}\frac{1}{p}%
\mathrm{d}t,
\end{equation*}%
thus we can show that $\left\{ y_{n}\left( x\right) \right\} _{n=m}^{\infty
} $ is equicontinuous on each $e\in \mathcal{E}\left( \Gamma _{m}\right) \ $%
if $\left\{ py_{n}^{\prime }\left( x\right) \right\} _{n=m}^{\infty }$ is
uniformly bounded on $e\in \mathcal{E}\left( \Gamma _{m}\right) .$ Note that
for each $e\in \mathcal{E}\left( \Gamma _{m}\right) ,$ there exists a point $%
x_{e}\in e$ such that
\begin{equation}
py_{n}^{\prime }(x_{e})=\frac{\int_{e}\frac{1}{p}\cdot py_{n}^{\prime }%
\mathrm{d}t}{\int_{e}\frac{1}{p}\mathrm{d}t}\leq \frac{2\max\limits_{x\in
e}\left\vert y_{n}(x)\right\vert }{\int_{e}\frac{1}{p}\mathrm{d}t}\leq \frac{%
2C_{m1}}{\int_{e}\frac{1}{p}\mathrm{d}t}.  \label{999}
\end{equation}%
Since $ly_{n}=\lambda y_{n},$ for $x\in e,$ one has%
\begin{equation*}
|py_{n}^{\prime }(x)-py_{n}^{\prime }(x_{e})|\leq \int_{e}\left( \left\vert
q\right\vert +|\lambda |w\right) |y_{n}(t)|\mathrm{d}t
\end{equation*}%
and thus%
\begin{equation*}
|py_{n}^{\prime }(x)|\leq |py_{n}^{\prime }(x_{e})|+\int_{e}\left(
\left\vert q\right\vert +|\lambda |w\right) |y_{n}(t)|\mathrm{d}t.
\end{equation*}%
Then the uniform boundedness of $\left\{ py_{n}^{\prime }\left( x\right)
\right\} _{n=m}^{\infty }$ on each $e\in \mathcal{E}\left( \Gamma
_{m}\right) $ follows from $\left( \ref{999}\right) $ and the uniform
boundedness of $\left\{ y_{n}\left( x\right) \right\} _{n=m}^{\infty }$ on $%
\Gamma _{m}$. Moreover, for each $e\in \mathcal{E}\left( \Gamma _{m}\right) $
and $x_{1},x_{2}\in e,$%
\begin{equation*}
\left\vert py_{n}^{\prime }\left( x_{1}\right) -py_{n}^{\prime }\left(
x_{2}\right) \right\vert \leq \int_{x_{1}}^{x_{2}}\left( \left\vert
q\right\vert +|\lambda |w\right) |y_{n}(t)|\mathrm{d}t,
\end{equation*}%
which yields the equicontinuous of $\left\{ py_{n}^{\prime }\left( x\right)
\right\} _{n=m}^{\infty }$ on each $e\in \mathcal{E}\left( \Gamma
_{m}\right) .$

Then it follows from Arzela-Ascoli Theorem that there exists subsequence $%
\left\{ y_{n_{j}}\right\} \subset \left\{ y_{n}\right\} _{n=m}^{\infty }$
such that $\left\{ y_{n_{j}}\right\} $ and $\left\{ py_{n_{j}}^{\prime
}\right\} $ are uniformly convergent on each edge of $\Gamma _{m}$.

By diagonalization, we can extract a subsequence $\left\{ y_{n_{j,j}}\left(
x\right) \right\} $ and edgewise continuous functions $f$ and $g$ defined on
$\Gamma $ such that for all $\Gamma _{m}\subset \Gamma ,$
\begin{equation*}
\sup_{x\in \Gamma _{m}}\left\vert y_{n_{j,j}}-f\right\vert +\sup_{x\in
\Gamma _{m}}\left\vert py_{n_{j,j}}^{\prime }-g\right\vert \rightarrow 0,%
\text{ as }j\rightarrow \infty .
\end{equation*}%
Based the above considerations, now we aim to prove that the function $f$ is
a positive solution of the equation $ly=\lambda y$ on $\Gamma $. For each $%
e\in \mathcal{E}\left( \Gamma \right) ,$%
\begin{equation*}
y_{n_{j}}\left( x\right) -y_{n_{j}}\left( o\left( e\right) \right)
=\int_{o\left( e\right) }^{x}\frac{1}{p}\cdot py_{n_{j}}^{\prime }\mathrm{d}%
t.
\end{equation*}%
Then it follows from the dominated convergence theorem that%
\begin{equation*}
f\left( x\right) -f\left( o\left( e\right) \right) =\int_{o\left( e\right)
}^{x}\frac{1}{p}g\mathrm{d}t,
\end{equation*}%
which yields that $pf^{\prime }=g.$ Also,
\begin{equation*}
py_{n_{j,j}}^{\prime }\left( x\right) -py_{n_{j,j}}^{\prime }\left( o\left(
e\right) \right) =\int_{o\left( e\right) }^{x}\left( q-\lambda w\right)
y_{n_{j,j}}(t)\mathrm{d}t,
\end{equation*}%
which yields that $g^{\prime }=\left( q-\lambda w\right) f.$ Clearly, $f$
satisfies the Kirchhoff conditions at inner vertices. Moreover, it is
immediately seen from $\left( \ref{harnack}\right) $ that $f>0$ on $\Gamma .$
The proof is completed.
\end{proof}

Define the Dirichlet operator $\mathbf{H}_{\text{D}}\ $as follows:%
\begin{eqnarray}
\mathbf{H}_{\text{D}}f &=&lf,  \label{HD} \\
\text{dom}(\mathbf{H}_{\text{D}}) &=&\{f\in D_{\max }:\text{ }f|_{\partial
\Gamma }=0,\text{\ }f\text{ satisfies }(\ref{kirchhoff})\text{ at inner
vertices}\}.  \notag
\end{eqnarray}%
Here the maximal domain $D_{\max }$ contains functions $f\in
L_{w}^{2}(\Gamma )$ edgewise absolutely continuous and $lf\in
L_{w}^{2}(\Gamma )$. In addition, define the pre-minimal operator $\mathbf{H}%
^{0}y:=ly$ and
\begin{equation*}
\text{dom}(\mathbf{H}^{0}):=\{f\in \text{dom}(\mathbf{H}_{\text{D}}):f\text{
\ has compact support in }\Gamma \}.
\end{equation*}

\begin{remark}
\label{self}Assume that $1/p,\ w\in L_{loc}^{1}\left( \Gamma \right) ,$ $%
q\geq 0$, then $\mathbf{H}_{\text{\emph{D}}}$ is self-adjoint and coincides
with the operator $\mathbf{H}_{\mathbf{t}_{q}}$ associated with $\mathbf{t}%
_{q}$. In fact, in this case, the self-adjointness of $\mathbf{H}_{\text{%
\emph{D}}}$ follows from \cite[Remark 5.3]{gla}. Notice that if $f\in $ dom$(%
\mathbf{H}^{0})$, then $f\in $ dom$(\mathbf{t}_{q}^{0})$ with $\mathbf{t}%
_{q}^{0}\left[ f\right] =\left( \mathbf{H}^{0}f,f\right) .$ Moreover, for
every $u\in $ dom$\left( \mathbf{H}^{0}\right) $ and every $v\in $ dom$(%
\mathbf{t}_{q}^{0}),$ one has%
\begin{equation*}
\mathbf{t}_{q}\mathbf{[}u,v\mathbf{]=}\left( \mathbf{H}^{0}u,v\right) ,
\end{equation*}%
then it follows from \cite[Chapter 6, Corollary 2.4]{kato} that $\mathbf{H}%
^{0}\subset \mathbf{H}_{\mathbf{t}_{q}}.$ This implies that $\mathbf{H}_{%
\emph{D}}=\mathbf{H}_{\mathbf{t}_{q}}.$
\end{remark}

\section{Persson-type theorem on metric graphs}

In this section, we illustrate that under the following Hypothesis \ref{hyp
copy(1)}, the Persson-type theorem for $\mathbf{H}_{\mathbf{t}_{q}}$ can be
given, see Theorem \ref{persson}. Here and thereafter we use the following
notation $a_{+}$ $=\max \left\{ a,0\right\} $ and $a_{-}$ $=-\min \left\{
a,0\right\} $.

\begin{Hypothesis}
\label{hyp copy(1)}$(1)$ $1/p\in L^{\eta }\left( \Gamma \right) ,$ $\eta \in %
\left[ 1,+\infty \right] ,$ $q\in L_{loc}^{1}\left( \Gamma \right) ,$ $w\in
L_{loc}^{1}\left( \Gamma \right) ;$

$\left( 2\right) $ there exists a compact subgraph $\Gamma ^{\prime }\subset
\Gamma $ such that%
\begin{equation*}
C_{w}:=\underset{x\in \Gamma \backslash \Gamma ^{\prime }}{\text{\emph{ess}}%
\inf }w>0;
\end{equation*}

$(3)$ $\inf\limits_{e\in \mathcal{E}\left( \Gamma \right) }\left\vert
e\right\vert =d_{\ast }>0;$ $(4)$ $C_{q}:=\sup\limits_{e\in \mathcal{E}%
\left( \Gamma \right) }\int\nolimits_{e}q_{-}\mathrm{d}t<+\infty .$
\end{Hypothesis}

\begin{notation}
Fix a vertex $o\in \Gamma .$ Throughout this section, for any integer $n>0$,
let $\Gamma _{n}\subset \Gamma $ be the union of all edges $e$ such that
both endpoints of $e$ are at a distance at most $n$ from $o$.
\end{notation}

To give the Persson-type theorem, let us introduce the form $\mathbf{t}%
_{q,n}^{0}$ as a form sum of the two forms $\mathbf{t}_{0,n}^{0}$ and $%
\mathbf{q}_{n}\mathbf{:}$%
\begin{equation*}
\mathbf{t}_{q,n}^{0}\left[ f\right] =\mathbf{t}_{0,n}^{0}\left[ f\right] +%
\mathbf{q}_{n}\left[ f\right] ,\text{ dom}\left( \mathbf{t}_{q,n}^{0}\right)
=\text{dom}(\mathbf{t}_{0,n}^{0})\cap \text{dom}(\mathbf{q}_{n}),
\end{equation*}%
where
\begin{eqnarray*}
\mathbf{t}_{0,n}^{0}\left[ f\right] &=&\int\nolimits_{\Gamma \backslash
\Gamma _{n}}p(x)\left\vert f^{\prime }(x)\right\vert ^{2}\mathrm{d}x,\text{ }%
\mathbf{q}_{n}\left[ f\right] =\int\nolimits_{\Gamma \backslash \Gamma
_{n}}q\left\vert f(x)\right\vert ^{2}\mathrm{d}x, \\
\text{dom}(\mathbf{t}_{0,n}^{0}) &=&\{f\in H_{c}^{1}(\overline{\Gamma
\backslash \Gamma _{n}};p,w):f|_{\partial \left( \overline{\Gamma \backslash
\Gamma _{n}}\right) }=0,\text{ }f|_{\overline{\Gamma \backslash \Gamma _{n}}%
\cap \Gamma _{n}}=0\}, \\
\text{dom}(\mathbf{q}_{n}) &=&\{f\in L_{w}^{2}(\Gamma \backslash \Gamma
_{n}):\left\vert \mathbf{q}_{n}\left[ f\right] \right\vert <\infty \}.
\end{eqnarray*}%
If $\mathbf{t}_{q,n}^{0}$ is bounded from below, denote $\mathbf{H}_{\mathbf{%
t}_{q,n}}$ the self-adjoint operator associated with $\mathbf{t}_{q,n}:=%
\overline{\mathbf{t}_{q,n}^{0}}.$

\begin{theorem}[Persson-type theorem]
\label{persson}Suppose Hypothesis \ref{hyp copy(1)} holds$.$ Then
\begin{equation*}
\inf \sigma _{ess}\left( \mathbf{H}_{\mathbf{t}_{q}}\right)
=\lim_{n\rightarrow \infty }\inf \sigma \left( \mathbf{H}_{\mathbf{t}%
_{q,n}}\right) .
\end{equation*}
\end{theorem}

Before proving the Persson-type theorem, we need some preliminary lemmas and
notations.

\begin{lemma}
\label{sup}Suppose the conditions $(1)-(3)$ in Hypothesis \ref{hyp copy(1)}
are satisfied$.$ For every $\epsilon >0,$ there exists a constant $%
C_{\epsilon }$ such that for all $e\in \mathcal{E}\left( \Gamma \right) ,$
\begin{equation}
\sup\limits_{x\in e}\left\vert f(x)\right\vert ^{2}\leq \epsilon
\int\nolimits_{e}p(x)\left\vert f^{\prime }(x)\right\vert ^{2}\mathrm{d}%
x+C_{\epsilon }\int\nolimits_{e}w\left\vert f(x)\right\vert ^{2}\mathrm{d}x;
\end{equation}
\end{lemma}

\begin{proof}
Since $1/p\in L^{\eta }\left( \Gamma \right) $ for some $\eta \in \left[
1,+\infty \right] ,$ for every $\epsilon >0,$ there exists $\delta >0$ such
that for all $e\in \mathcal{E}\left( \Gamma \right) $ and $x\in e,$ one has%
\begin{equation}
\int\nolimits_{e\cap \Gamma (x;\delta )}\frac{1}{p(t)}\mathrm{d}t<\frac{%
\epsilon }{2},  \label{pdelta}
\end{equation}%
where $\Gamma (x;\delta )=\left\{ y\in \Gamma \left\vert \rho \left(
x,y\right) \leq \delta \right. \right\} .$ We can assume that $\delta <\frac{%
d_{\ast }}{2}.$ In fact, it is easy to prove the case when $\eta =1$ or $%
\eta =+\infty ;$ for $\eta \in \left( 1,+\infty \right) ,$ this can be seen
with the help of the Holder inequality. For instance, if $1/p\in L^{2}\left(
\Gamma \right) ,$ for every $\epsilon >0,$ there exists $\delta \left(
\epsilon \right) >0$ such that for all $x\in \Gamma ,$
\begin{equation*}
\int\nolimits_{e\cap \Gamma (x;\delta )}\frac{1}{p(t)}\mathrm{d}t\leq \left(
\int\nolimits_{e\cap \Gamma (x;\delta )}1^{2}\mathrm{d}t\int\nolimits_{%
\Gamma (x;\delta )}\frac{1}{p^{2}(t)}\mathrm{d}t\right) ^{\frac{1}{2}}\leq
d_{\ast }\cdot \left( \int\nolimits_{\Gamma }\frac{1}{p^{2}(t)}\mathrm{d}%
t\right) ^{\frac{1}{2}}<\frac{\epsilon }{2}.
\end{equation*}%
Under the condition on $w,$ there exists $c>0$ such that for all $x\in
\Gamma ,$
\begin{equation}
\int\nolimits_{\Gamma (x;\frac{\delta }{2})}w(t)\mathrm{d}t>c.
\end{equation}%
For $f\in H^{1}(\Gamma ;p,w)\ $and $x,y\in e\in \mathcal{E}\left( \Gamma
\right) ,$ one has%
\begin{eqnarray}
\left\vert f\left( x\right) \right\vert ^{2} &\leq &2\left\vert f\left(
y\right) \right\vert ^{2}+2\left\vert f\left( x\right) -f\left( y\right)
\right\vert ^{2}=2\left\vert f\left( y\right) \right\vert ^{2}+2\left\vert
\int\nolimits_{x}^{y}f^{\prime }\left( t\right) \mathrm{d}t\right\vert ^{2}
\label{A2} \\
&\leq &2\left\vert f\left( y\right) \right\vert
^{2}+2\int\nolimits_{x}^{y}p\left( t\right) \left\vert f^{\prime }\left(
t\right) \right\vert ^{2}\mathrm{d}t\int\nolimits_{x}^{y}\frac{1}{p(t)}%
\mathrm{d}t.  \notag
\end{eqnarray}%
We multiply $\left( \ref{A2}\right) $ with $w(y)$ and integrate over $%
I(x,e)=e\cap \Gamma (x;\delta )$ for arbitrary $e\in \mathcal{E}\left(
\Gamma \right) $ and $x\in e.$ Note that the volume of $I(x,e)$ is $\geq
\delta .$ Then
\begin{eqnarray}
&&\left\vert f\left( x\right) \right\vert ^{2}\int\nolimits_{I(x,e)}w(y)dy \\
&\leq &2\int\nolimits_{I(x,e)}\left\vert f\left( y\right) \right\vert
^{2}w(y)dy+2\int\nolimits_{I(x,e)}w(y)\left( \int\nolimits_{x}^{y}p\left(
t\right) \left\vert f^{\prime }\left( t\right) \right\vert ^{2}\mathrm{d}%
t\int\nolimits_{x}^{y}\frac{1}{p(t)}\mathrm{d}t\right) dy  \notag \\
&\leq &2\int\nolimits_{I(x,e)}\left\vert f\left( y\right) \right\vert
^{2}w(y)dy+2\int\nolimits_{I(x,e)}w(y)dy\int\nolimits_{e}p\left( y\right)
\left\vert f^{\prime }\left( y\right) \right\vert
^{2}dy\int\nolimits_{I(x,e)}\frac{1}{p(y)}dy  \notag
\end{eqnarray}%
Dividing this by $\int\nolimits_{I(x,e)}w(y)dy$ and defining $C_{\epsilon }=%
\frac{2}{c},$ one proves the statement$.$
\end{proof}

\begin{lemma}
\label{sup copy(1)}Suppose Hypothesis \ref{hyp copy(1)} holds$.$ Then for
every $\epsilon >0,$ there exists a constant $C_{\epsilon }$ such that
\begin{equation}
\int\nolimits_{\Gamma }q_{-}\left\vert f(x)\right\vert ^{2}\mathrm{d}x\leq
\epsilon \int\nolimits_{\Gamma }p(x)\left\vert f^{\prime }(x)\right\vert ^{2}%
\mathrm{d}x+C_{\epsilon }\int\nolimits_{\Gamma }w\left\vert f(x)\right\vert
^{2}\mathrm{d}x  \label{q-}
\end{equation}%
for all $f\in H^{1}(\Gamma ;p,w).$ Moreover, the form $\mathbf{t}_{q}^{0}$
is lower semibounded.
\end{lemma}

\begin{proof}
The claim $\left( \ref{q-}\right) $ is a direct consequence of Lemma \ref%
{sup} in virtue of
\begin{equation*}
\int\nolimits_{\Gamma }q_{-}\left\vert f(x)\right\vert ^{2}\mathrm{d}%
x=\sum\limits_{e\in \mathcal{E}\left( \Gamma \right)
}\int\nolimits_{e}q_{-}\left\vert f(x)\right\vert ^{2}\mathrm{d}x\leq \left(
\sup_{e\in \mathcal{E}\left( \Gamma \right) }\int\nolimits_{e}q_{-}\mathrm{d}%
t\right) \sum\limits_{e\in \mathcal{E}\left( \Gamma \right) }\left\Vert
f^{2}\right\Vert _{L^{\infty }\left( e\right) }.
\end{equation*}%
Taking $\left( \ref{q-}\right) $ into account and letting $\epsilon =\frac{1%
}{2},$ we obtain%
\begin{equation*}
\mathbf{t}_{q}^{0}\left[ f\right] =\int\nolimits_{\Gamma }p(x)\left\vert
f^{\prime }(x)\right\vert ^{2}+q_{+}\left\vert f(x)\right\vert ^{2}\mathrm{d}%
x-\int\nolimits_{\Gamma }q_{-}\left\vert f(x)\right\vert ^{2}\mathrm{d}x\geq
-C_{\frac{1}{2}}\int\nolimits_{\Gamma }w\left\vert f(x)\right\vert ^{2}%
\mathrm{d}x,
\end{equation*}%
which implies that the form $\mathbf{t}_{q}^{0}$ is lower semibounded$.$
\end{proof}

\begin{remark}
Suppose Hypothesis \ref{hyp copy(1)} holds$.\ $Then Lemma \ref{sup copy(1)}
also holds for the graphs $\overline{\Gamma \backslash \Gamma _{n}}$
according to the definition of $\Gamma _{n}.$
\end{remark}

\begin{lemma}
\label{s close}Suppose Hypothesis \ref{hyp copy(1)} holds$.$ The following
form
\begin{eqnarray*}
\mathbf{s}_{q}^{0}\left[ f\right] &=&\int\nolimits_{\Gamma }p(x)\left\vert
f^{\prime }(x)\right\vert ^{2}+q\left\vert f(x)\right\vert ^{2}\mathrm{d}x,%
\text{ } \\
\text{dom}(\mathbf{s}_{q}^{0}) &=&\{f\in H^{1}(\Gamma ;p,w):\left\vert
\mathbf{q}\left[ f\right] \right\vert <\infty ,\text{ }f|_{\partial \Gamma
}=0\}\text{ }
\end{eqnarray*}%
is lower semibounded and closed$.$ Moreover, $\mathbf{t}_{q}\subset \mathbf{s%
}_{q}^{0}.$
\end{lemma}

\begin{proof}
Define%
\begin{eqnarray*}
&&\mathbf{q}_{-}:=-\int\nolimits_{\Gamma }q_{-}\left\vert f(x)\right\vert
^{2}\mathrm{d}x, \\
&&\text{dom}(\mathbf{q}_{-}):=\{f\in L_{w}^{2}(\Gamma ):\left\vert \mathbf{q}%
_{-}\left[ f\right] \right\vert <\infty \}.
\end{eqnarray*}%
It follows from Lemma \ref{sup copy(1)} that the form $\mathbf{q}_{-}$ is
infinitesimally $\mathbf{s}_{q_{+}}^{0}$ bounded. It is obvious that $%
\mathbf{s}_{q_{+}}^{0}$ is lower semibounded and closed. Applying KLMN
theorem $\cite{kato}$, we complete the proof$.$
\end{proof}

\begin{remark}
\label{s close copy(1)}Suppose Hypothesis \ref{hyp copy(1)} holds$.$ The
forms
\begin{eqnarray*}
\mathbf{s}_{q,n}^{0}\left[ f\right] &=&\int\nolimits_{\Gamma \backslash
\Gamma _{n}}p(x)\left\vert f^{\prime }(x)\right\vert ^{2}+q\left\vert
f(x)\right\vert ^{2}\mathrm{d}x, \\
\text{dom}(\mathbf{s}_{q,n}^{0}) &=&\{f\in H^{1}(\Gamma \backslash \Gamma
_{n};p,w):\left\vert \mathbf{q}_{n}\left[ f\right] \right\vert <\infty ,%
\text{ }f|_{\partial \left( \overline{\Gamma \backslash \Gamma _{n}}\right)
}=0,\text{ }f|_{\overline{\Gamma \backslash \Gamma _{n}}\cap \Gamma _{n}}=0\}
\end{eqnarray*}%
are lower semibounded and closed$.$ Moreover, $\mathbf{t}_{q,n}\subset
\mathbf{s}_{q,n}^{0}.$
\end{remark}

\begin{lemma}
\label{persson reverse} Let $\mathbf{s}$ be a closed quadratic form on $%
L_{w}^{2}\left( \Gamma \right) $ that is bounded from below and let $\mathbf{%
H}$ be the corresponding self-adjoint operator. Assume that there is a
normalized sequence $f_{n}$ in dom$\left( \mathbf{s}\right) $ that converges
weakly to zero. Then%
\begin{equation*}
\inf \sigma _{ess}\left( \mathbf{H}\right) \leq \underset{n\rightarrow
\infty }{\lim \inf }\text{ }\mathbf{s[}f_{n}\mathbf{].}
\end{equation*}
\end{lemma}

\begin{proof}
see $\cite{Haeseler S}$.
\end{proof}

Now we in a position to prove the Persson-type Theorem for the operator $%
\mathbf{H}_{\mathbf{t}_{q}}$.

\begin{proof}[Proof of Theorem \protect\ref{persson}]
Firstly, we prove that
\begin{equation}
\inf \sigma _{ess}\left( \mathbf{H}_{\mathbf{t}_{q}}\right) \geq
\lim\limits_{n\rightarrow \infty }\inf \sigma \left( \mathbf{H}_{\mathbf{t}%
_{q,n}}\right) =:l.  \label{zuo}
\end{equation}%
For any $\lambda \in \sigma _{ess}\left( \mathbf{H}_{\mathbf{t}_{q}}\right)
, $ we shall prove that $\lambda \geq l.$ From Weyl theorem, one can choose
a sequence $\left\{ u_{m}\right\} _{m=1}^{\infty }\subset $ $\mathrm{dom}%
\left( \mathbf{H}_{\mathbf{t}_{q}}\right) $ such that
\begin{equation}
\left\Vert u_{m}\right\Vert _{L_{w}^{2}(\Gamma )}=1\text{ for all }m,
\label{um1}
\end{equation}%
\begin{equation*}
u_{m}\overset{w}{\longrightarrow }\text{ 0 as }m\rightarrow \infty ,
\end{equation*}%
\begin{equation*}
\left\Vert \mathbf{H}_{\mathbf{t}_{q}}u_{m}-\lambda u_{m}\right\Vert
_{L_{w}^{2}(\Gamma )}\rightarrow 0\text{ as }m\rightarrow \infty .
\end{equation*}%
Denote%
\begin{equation*}
\tilde{\Gamma}_{n}:=\Gamma _{n}\cup \left\{ e\in \mathcal{E}\left( \Gamma
\right) :e\text{ is an edge with only one vertex }v\text{ in }\Gamma
_{n}\right\} .
\end{equation*}%
Then $\Gamma \backslash \tilde{\Gamma}_{n}$ is the union of all edges which
do not have vertices in $\Gamma _{n}.$ Now define functions $\varphi _{n}$
on $\Gamma $ such that%
\begin{equation}
\varphi _{n}=0\text{ for }x\in \Gamma _{n},  \label{0}
\end{equation}%
\begin{equation}
\varphi _{n}=1\text{ for }x\in \Gamma \backslash \tilde{\Gamma}_{n},
\label{1}
\end{equation}%
\begin{equation}
0\leq \varphi _{n}\leq 1\text{, }  \label{01}
\end{equation}%
\begin{equation*}
\frac{\sqrt{p}}{\sqrt{w}}\varphi _{n}^{\prime }\text{ bounded in }\Gamma .
\end{equation*}%
Let $e$ be an edge with only one vertex $v_{e}$ in $\Gamma _{n}.$ Without
loss of generality, assume that $v_{e}$ is the initial vertex of $e.$ Then
for $x\in e,$ we can define $\varphi _{n}\left( x\right) =1-\frac{%
\int_{x}^{t\left( e\right) }\frac{\sqrt{w}}{\sqrt{p}}\mathrm{d}t}{\int_{e}%
\frac{\sqrt{w}}{\sqrt{p}}\mathrm{d}t}.$

Now denote $f_{m,n}:$ $=\varphi _{n}u_{m}.$ We are going to prove that $%
\left\{ f_{m,n}\right\} _{m=1}^{\infty }\subset $ dom$(\mathbf{t}_{q,n})$
and as $m\rightarrow \infty ,$%
\begin{equation}
\mathbf{t}_{q,n}\left[ f_{m,n}\right] \leq \lambda +o(1)\text{ }  \label{tnf}
\end{equation}%
and
\begin{equation}
\left\Vert f_{m,n}\right\Vert _{L_{w}^{2}(\Gamma \backslash \Gamma
_{n})}^{2}=1+o(1).  \label{fn}
\end{equation}%
We observe that for any $u_{m}\in $ dom$\left( \mathbf{H}_{\mathbf{t}%
_{q}}\right) ,$ there exists a sequence $\left\{ g_{k,m}\right\}
_{k=1}^{\infty }\subset $\textrm{\ $\mathrm{dom}$}$\left( \mathbf{t}%
_{q}^{0}\right) $ such that%
\begin{equation}
\left\Vert g_{k,m}-u_{m}\right\Vert _{\mathbf{t}_{q}}\rightarrow 0\text{ as }%
k\rightarrow \infty .  \label{guquyuling}
\end{equation}%
In order to prove $\left\{ f_{m,n}\right\} _{m=1}^{\infty }\in $ dom$(%
\mathbf{t}_{q,n}),$ it is enough to prove that $\varphi _{n}g_{k,m}\in $%
\textrm{\ $\mathrm{dom}$}$\left( \mathbf{t}_{q,n}^{0}\right) $ and $%
\left\Vert \varphi _{n}g_{k,m}-\varphi _{n}u_{m}\right\Vert _{\mathbf{t}%
_{q,n}}\rightarrow 0$ as $k\rightarrow \infty .$ In fact, by use of the
properties of $\varphi _{n}$ and the fact $\left\{ g_{k,m}\right\}
_{k=1}^{\infty }\subset $ dom$\left( \mathbf{t}_{q}^{0}\right) ,$ we get $%
\left\{ \varphi _{n}g_{k,m}\right\} _{k=1}^{\infty }\subset $ dom$\left(
\mathbf{q}_{n}\right) $ and
\begin{eqnarray*}
\int\nolimits_{\Gamma \backslash \Gamma _{n}}p\left\vert \left( \varphi
_{n}g_{k,m}\right) ^{\prime }\right\vert ^{2}\mathrm{d}x &\leq &2\left(
\int\nolimits_{\Gamma \backslash \Gamma _{n}}p\left\vert \varphi
_{n}g_{k,m}^{\prime }\right\vert ^{2}\mathrm{d}x+\int\nolimits_{\Gamma
\backslash \Gamma _{n}}p\left\vert \varphi _{n}^{\prime }g_{k,m}\right\vert
^{2}\mathrm{d}x\right) \\
&\leq &C\left( \int\nolimits_{\Gamma \backslash \Gamma _{n}}p\left\vert
g_{k,m}^{\prime }\right\vert ^{2}\mathrm{d}x+\int\nolimits_{\Gamma
\backslash \Gamma _{n}}\left\vert g_{k,m}\right\vert ^{2}w\mathrm{d}x\right)
\end{eqnarray*}%
where $C$ is some constant, which implies that $\varphi _{n}g_{k,m}\in $ dom$%
\left( \mathbf{t}_{q,n}^{0}\right) .$ From Lemma \ref{sup copy(1)}, \ it
follows that%
\begin{eqnarray*}
&&\left\Vert g_{k,m}-u_{m}\right\Vert _{\mathbf{t}_{q}} \\
&=&\mathbf{t}_{q}\left[ g_{k,m}-u_{m}\right] +\left( c+1\right)
\int\nolimits_{\Gamma }w\left\vert g_{k,m}-u_{m}\right\vert ^{2}\mathrm{d}x
\\
&\geq &\frac{1}{2}\int\nolimits_{\Gamma }p\left\vert \left(
g_{k,m}-u_{m}\right) ^{\prime }\right\vert ^{2}\mathrm{d}x+\int\nolimits_{%
\Gamma }q_{+}\left\vert g_{k,m}-u_{m}\right\vert ^{2}\mathrm{d}x \\
&&-C_{\frac{1}{2}}\int\nolimits_{\Gamma }w\left\vert
g_{k,m}-u_{m}\right\vert ^{2}\mathrm{d}x+\int\nolimits_{\Gamma }\left(
c+1\right) w\left\vert g_{k,m}-u_{m}\right\vert ^{2}\mathrm{d}x,
\end{eqnarray*}%
where $c>0$ is some positive constant such that $\mathbf{t}_{q}\geq -c.$
Therefore, we see from $\left( \ref{guquyuling}\right) $ and Lemma \ref{sup
copy(1)} that the expressions
\begin{equation*}
\int\nolimits_{\Gamma }p\left\vert \left( g_{k,m}-u_{m}\right) ^{\prime
}\right\vert ^{2}\mathrm{d}x,\text{ }\int\nolimits_{\Gamma }w\left\vert
g_{k,m}-u_{m}\right\vert ^{2}\mathrm{d}x\text{ and }\int\nolimits_{\Gamma
}\left\vert q\right\vert \left\vert g_{k,m}-u_{m}\right\vert ^{2}\mathrm{d}x
\end{equation*}%
all tend to zero as $k\rightarrow \infty .$ Hence%
\begin{eqnarray*}
&&\left\Vert \varphi _{n}g_{k,m}-\varphi _{n}u_{m}\right\Vert _{\mathbf{t}%
_{q,n}} \\
&=&\int\nolimits_{\Gamma \backslash \Gamma _{n}}p\left\vert \left( \varphi
_{n}\left( g_{k,m}-u_{m}\right) \right) ^{\prime }\right\vert ^{2}+\left(
q+\left( c_{n}+1\right) w\right) \left\vert \varphi _{n}\left(
g_{k,m}-u_{m}\right) \right\vert ^{2}\mathrm{d}x \\
&\leq &C\left( \int\nolimits_{\Gamma }p\left\vert \left(
g_{k,m}-u_{m}\right) ^{\prime }\right\vert ^{2}\mathrm{d}x+\int\nolimits_{%
\Gamma }w\left\vert g_{k,m}-u_{m}\right\vert ^{2}\mathrm{d}x\right) \\
&&+\int\nolimits_{\Gamma }\left( \left\vert q\right\vert +\left(
c_{n}+1\right) w\right) \left\vert g_{k,m}-u_{m}\right\vert ^{2}\mathrm{d}%
x\rightarrow 0,\text{ as }k\rightarrow \infty
\end{eqnarray*}%
where $c_{n}>0$ are positive constants such that $\mathbf{t}_{q,n}\geq
-c_{n}.$ This proves $\left\{ f_{m,n}\right\} _{m=1}^{\infty }\subset $ dom$(%
\mathbf{t}_{q,n}).$

Next, we aim to prove $\left( \ref{tnf}\right) .$ In virtue of the
properties of $\left\{ u_{m}\right\} _{m=1}^{\infty },$
\begin{eqnarray*}
\mathbf{t}_{q}\left[ u_{m}\right] -\lambda &=&(\mathbf{H}_{\mathbf{t}%
_{q}}u_{m},u_{m})-\lambda \leq \left\Vert \mathbf{H}_{\mathbf{t}%
_{q}}u_{m}\right\Vert _{L_{w}^{2}(\Gamma )}\left\Vert u_{m}\right\Vert
_{L_{w}^{2}(\Gamma )}-\lambda \left\Vert u_{m}\right\Vert _{L_{w}^{2}(\Gamma
)} \\
&\leq &\left\Vert \mathbf{H}_{\mathbf{t}_{q}}u_{m}-\lambda u_{m}\right\Vert
_{L_{w}^{2}(\Gamma )},
\end{eqnarray*}%
and thus%
\begin{equation}
\mathbf{t}_{q}\left[ u_{m}\right] \leq \lambda +o(1).  \label{tqlamuda}
\end{equation}%
Therefore, it follows from Lemma \ref{sup copy(1)} that as $m\rightarrow
\infty ,$
\begin{eqnarray}
&&\int\nolimits_{\Gamma }p\left\vert u_{m}^{\prime }\right\vert ^{2}\mathrm{d%
}x\leq \lambda +\int\nolimits_{\Gamma }q_{-}\left\vert u_{m}\right\vert ^{2}%
\mathrm{d}x+o(1)  \label{py} \\
&\leq &\lambda +\frac{1}{2}\int\nolimits_{\Gamma }p\left\vert u_{m}^{\prime
}\right\vert ^{2}\mathrm{d}x+C_{\frac{1}{2}}\int\nolimits_{\Gamma
}w\left\vert u_{m}\right\vert ^{2}\mathrm{d}x+o(1)
\end{eqnarray}%
which implies that%
\begin{equation}
\int\nolimits_{\Gamma }p\left\vert u_{m}^{\prime }\right\vert ^{2}\mathrm{d}%
x\leq 2\left( \lambda +C_{\frac{1}{2}}+o(1)\right) .  \label{p bound}
\end{equation}

We also observe that for fixed $n,$%
\begin{equation}
\int\nolimits_{\Gamma _{n}}\left\vert u_{m}\right\vert ^{2}w\mathrm{d}%
x\rightarrow 0\text{ as }m\rightarrow \infty .  \label{uw}
\end{equation}

Let $e_{o}$ be an edge incident to $o.$ Then $u_{m}(o)=u_{m}(t)-%
\int_{o}^{t}u_{m}^{\prime }(s)\mathrm{d}s\ $for $t\in e_{o}.$ Multiply by $%
w(t)$ and integrate over $e_{o},$
\begin{equation*}
u_{m}(o)\int_{e_{o}}w(t)\mathrm{d}t=\int_{e_{o}}w(t)u_{m}(t)\mathrm{d}%
t-\int_{e_{o}}w(t)\int_{0}^{t}u_{m}^{\prime }(s)\mathrm{d}s\mathrm{d}t,
\end{equation*}%
then%
\begin{eqnarray*}
\left\vert u_{m}(o)\right\vert \int_{e_{o}}w(t)\mathrm{d}t &\leq &\left\vert
\int_{e_{o}}w(t)u_{m}(t)\mathrm{d}t\right\vert +\left\vert
\int_{e_{o}}w(t)\int_{0}^{t}u_{m}^{\prime }(s)\mathrm{d}s\mathrm{d}%
t\right\vert \\
&\leq &\left( \int_{e_{o}}w(t)\mathrm{d}t\int_{e_{o}}w(t)\left\vert
u_{m}(t)\right\vert ^{2}\mathrm{d}t\right) ^{\frac{1}{2}}+ \\
&&\left( \int_{e_{o}}w(t)\mathrm{d}t\int_{e_{o}}\frac{1}{p(t)}\mathrm{d}%
t\int_{e_{o}}p(t)\left\vert u_{m}^{\prime }(t)\right\vert ^{2}\mathrm{d}%
t\right) ^{\frac{1}{2}}.
\end{eqnarray*}%
Therefore, it follows from $\left( \ref{um1}\right) $ and $\left( \ref{p
bound}\right) $ that $\left\{ u_{m}(o)\right\} $ is bounded$.$ Moreover, for
any $x_{1}$, $x_{2}\in \Gamma _{n}$%
\begin{equation}
\left\vert u_{m}\left( x_{1}\right) -u_{m}\left( x_{2}\right) \right\vert
=\left\vert \int\nolimits_{x_{1}}^{x_{2}}u_{m}^{\prime }\left( t\right)
\mathrm{d}t\right\vert \leq \left( \int\nolimits_{x_{1}}^{x_{2}}\frac{1}{p}%
\mathrm{d}t\cdot \int\nolimits_{x_{1}}^{x_{2}}p\left\vert u_{m}^{\prime
}\right\vert ^{2}\mathrm{d}x\right) ^{\frac{1}{2}}.  \label{equicon}
\end{equation}%
This together with $\left( \ref{p bound}\right) $ and the boundedness of $%
\left\{ u_{m}(o)\right\} $ yields that $\left\{ u_{m}\right\} $ are
uniformly bounded and uniformly equicontinuous on $\Gamma _{n}.$ Then it
follows from Arzela-Ascoli theorem that there is a subsequence $\left\{
u_{m_{k}}\right\} ,$ which is convergent in $L_{w}^{2}(\Gamma _{n}).$ Since $%
u_{m_{k}}\overset{w}{\longrightarrow }0$ as $k\rightarrow \infty ,$ the
limit function must be zero, that is $\left\Vert u_{m_{k}}\right\Vert
_{L_{w}^{2}(\Gamma _{n})}\rightarrow 0$ as $k\rightarrow \infty .$ But then
the original sequence itself must have this property, since otherwise we
could get a contradiction by applying the arguments above to a suitable
subsequence. Hence $\left( \ref{uw}\right) $ is proved.

By use of the definition of $\varphi _{n},$ $\left( \ref{p bound}\right) $
and $\left( \ref{uw}\right) ,$ one has%
\begin{eqnarray}
&&\int\nolimits_{\Gamma \backslash \Gamma _{n}}p\left\vert \left( \varphi
_{n}u_{m}\right) ^{\prime }\right\vert ^{2}\mathrm{d}x  \label{ppp} \\
&\leq &\int\nolimits_{\Gamma \backslash \Gamma _{n}}p\left\vert \varphi
_{n}u_{m}^{\prime }\right\vert ^{2}\mathrm{d}x+2\int\nolimits_{\Gamma
\backslash \Gamma _{n}}p\left\vert \varphi _{n}u_{m}^{\prime }\right\vert
\left\vert \varphi _{n}^{\prime }u_{m}\right\vert \mathrm{d}%
x+\int\nolimits_{\Gamma \backslash \Gamma _{n}}p\left\vert \varphi
_{n}^{\prime }u_{m}\right\vert ^{2}\mathrm{d}x  \notag \\
&\leq &\int\nolimits_{\Gamma \backslash \Gamma _{n}}p\left\vert
u_{m}^{\prime }\right\vert ^{2}\mathrm{d}x+C\left( \int\nolimits_{\tilde{%
\Gamma}_{n}}p\left\vert u_{m}^{\prime }\right\vert ^{2}\mathrm{d}%
x\int\nolimits_{\tilde{\Gamma}_{n}}\left\vert u_{m}\right\vert ^{2}w\mathrm{d%
}x\right) ^{\frac{1}{2}}+C\int\nolimits_{\tilde{\Gamma}_{n}}\left\vert
u_{m}\right\vert ^{2}w\mathrm{d}x  \notag \\
&\leq &\int\nolimits_{\Gamma }p\left\vert u_{m}^{\prime }\right\vert ^{2}%
\mathrm{d}x+o(1)\text{ as }m\rightarrow \infty .  \notag
\end{eqnarray}%
From Lemma \ref{sup}, it follows that for every $\epsilon >0,$ there exists
a constant $C_{\epsilon }$ such that
\begin{eqnarray*}
\int\nolimits_{\tilde{\Gamma}_{n}}\left\vert q\right\vert \left\vert
u_{m}\right\vert ^{2}\mathrm{d}x &=&\sum\limits_{e\in \mathcal{E}\left(
\tilde{\Gamma}_{n}\right) }\int\nolimits_{e}\left\vert q\right\vert
\left\vert u_{m}\right\vert ^{2}\mathrm{d}x\leq \left( \sup_{e\in \mathcal{E}%
\left( \tilde{\Gamma}_{n}\right) }\int\nolimits_{e}\left\vert q\right\vert
\mathrm{d}t\right) \sum\limits_{e\in \mathcal{E}\left( \tilde{\Gamma}%
_{n}\right) }\left\Vert u_{m}^{2}\right\Vert _{L^{\infty }\left( e\right) }
\\
&\leq &\epsilon C_{n}\int\nolimits_{\tilde{\Gamma}_{n}}p\left\vert
u_{m}^{\prime }\right\vert ^{2}\mathrm{d}x+C_{\epsilon }C_{n}\int\nolimits_{%
\tilde{\Gamma}_{n}}w\left\vert u_{m}\right\vert ^{2}\mathrm{d}x
\end{eqnarray*}%
where $C_{n}:=\sup\limits_{e\in \mathcal{E}\left( \tilde{\Gamma}_{n}\right)
}\int\nolimits_{e}\left\vert q\right\vert \mathrm{d}t.$ This together with $%
\left( \ref{p bound}\right) $ and $\left( \ref{uw}\right) $ yields that
\begin{equation*}
\int\nolimits_{\tilde{\Gamma}_{n}}\left\vert q\right\vert \left\vert
u_{m}\right\vert ^{2}\mathrm{d}x=o(1)\text{ as }m\rightarrow \infty .
\end{equation*}%
Therefore,%
\begin{eqnarray}
&&\int\nolimits_{\Gamma \backslash \Gamma _{n}}q\left\vert \varphi
_{n}u_{m}\right\vert ^{2}\mathrm{d}x=\int\nolimits_{\tilde{\Gamma}%
_{n}\backslash \Gamma _{n}}q\left\vert \varphi _{n}u_{m}\right\vert ^{2}%
\mathrm{d}x+\int\nolimits_{\Gamma \backslash \tilde{\Gamma}_{n}}q\left\vert
u_{m}\right\vert ^{2}\mathrm{d}x  \label{qqq} \\
&=&\int\nolimits_{\tilde{\Gamma}_{n}}q\left\vert \varphi
_{n}u_{m}\right\vert ^{2}\mathrm{d}x+\int\nolimits_{\Gamma }q\left\vert
u_{m}\right\vert ^{2}\mathrm{d}x-\int\nolimits_{\tilde{\Gamma}%
_{n}}q\left\vert u_{m}\right\vert ^{2}\mathrm{d}x  \notag \\
&=&\int\nolimits_{\Gamma }q\left\vert u_{m}\right\vert ^{2}\mathrm{d}x+o(1)%
\text{ as }m\rightarrow \infty .  \notag
\end{eqnarray}%
This together with $\left( \ref{tqlamuda}\right) $ and $\left( \ref{ppp}%
\right) $ implies $\left( \ref{tnf}\right) .$

To prove $\left( \ref{fn}\right) ,$ we note from $\left( \ref{0}\right) $, $%
\left( \ref{1}\right) $, $\left( \ref{01}\right) $ and $\left( \ref{uw}%
\right) $ that
\begin{eqnarray*}
0 &\leq &\int\nolimits_{\Gamma }\left\vert u_{m}\right\vert ^{2}w\mathrm{d}%
x-\int\nolimits_{\Gamma \backslash \Gamma _{n}}\left\vert \varphi
_{n}u_{m}\right\vert ^{2}w\mathrm{d}x=\int\nolimits_{\Gamma }\left\vert
u_{m}\right\vert ^{2}w\mathrm{d}x-\int\nolimits_{\Gamma }\left\vert \varphi
_{n}u_{m}\right\vert ^{2}w\mathrm{d}x \\
&=&\int\nolimits_{\Gamma }\left\vert u_{m}\right\vert ^{2}w\mathrm{d}x-\left[
\int\nolimits_{\Gamma \backslash \tilde{\Gamma}_{n}}\left\vert \varphi
_{n}u_{m}\right\vert ^{2}w\mathrm{d}x+\int\nolimits_{\tilde{\Gamma}%
_{n}}\left\vert \varphi _{n}u_{m}\right\vert ^{2}w\mathrm{d}x\right] \\
&=&\int\nolimits_{\tilde{\Gamma}_{n}}\left\vert u_{m}\right\vert ^{2}w%
\mathrm{d}x-\int\nolimits_{\tilde{\Gamma}_{n}}\left\vert \varphi
_{n}u_{m}\right\vert ^{2}w\mathrm{d}x\rightarrow 0\text{ as }m\rightarrow
\infty .
\end{eqnarray*}%
Hence $\left( \ref{fn}\right) $ follows from $\left( \ref{um1}\right) $.

It follows from the definition of $l$ that for any given number $\epsilon
>0, $ there exists a positive number $N$ such that for all $n>N,$%
\begin{equation*}
\inf \sigma \left( \mathbf{H}_{\mathbf{t}_{q,n}}\right) \geq l-\epsilon ,%
\text{ i.e., }\inf\limits_{f\in \text{ dom}\left( \mathbf{t}_{q,n}\right) }%
\frac{\mathbf{t}_{q,n}\left[ f\right] }{\left\Vert f\right\Vert
_{L_{w}^{2}(\Gamma \backslash \Gamma _{n})}^{2}}\geq l-\epsilon .
\end{equation*}%
This yields that for every $f\in $ dom$\left( \mathbf{t}_{q,n}\right) ,$
\begin{equation*}
\mathbf{t}_{q,n}\left[ f\right] \geq \left( l-\epsilon \right) \left\Vert
f\right\Vert _{L_{w}^{2}(\Gamma \backslash \Gamma _{n})}^{2}.
\end{equation*}%
Combining this with $\left( \ref{tnf}\right) $ and $\left( \ref{fn}\right) $
we immediately get
\begin{equation*}
\left( l-\epsilon \right) \left[ 1+o(1)\right] \leq \mathbf{t}_{q,n}\left[
f_{m,n}\right] \leq \lambda +o(1)\text{ as }m\rightarrow \infty .
\end{equation*}%
Since $\epsilon $ is arbitrary, $\left( \ref{zuo}\right) $ is proved.

The reverse inequality follows from Lemma \ref{persson reverse}. In fact, we
can pick a sequence of functions $f_{n}$ $\in $ \textrm{{\ $\mathrm{dom}$}}$%
\left( \mathbf{t}_{q,n}\right) $ vanishing on $\Gamma _{n}$ and satisfying $%
\left\Vert f_{n}\right\Vert _{L_{w}^{2}(\Gamma )}^{2}=1$ such that%
\begin{equation*}
\left\vert \inf \sigma \left( \mathbf{H}_{\mathbf{t}_{q,n}}\right) -\mathbf{t%
}_{q,n}\left[ f_{n}\right] \right\vert \leq \frac{1}{n}
\end{equation*}%
for all $n\in
\mathbb{N}
.$ Then $\left\{ f_{n}\right\} $ converges weakly to zero. Moreover, by
construction
\begin{equation*}
\underset{n\rightarrow \infty }{\lim }\mathbf{t}_{q}\mathbf{[}f_{n}\mathbf{]}%
=\underset{n\rightarrow \infty }{\lim }\mathbf{t}_{q,n}[f_{n}]\mathbf{=}%
\underset{n\rightarrow \infty }{\lim }\inf \sigma \left( \mathbf{H}_{\mathbf{%
t}_{q,n}}\right) \mathbf{.}
\end{equation*}%
Now Lemma \ref{persson reverse} gives the desired inequality.
\end{proof}

\begin{remark}
With slightly modifications, Theorem \ref{persson} can be extended to the
case when we only assume $1/p,\ w\in L_{loc}^{1}\left( \Gamma \right) \ $and
$q\geq 0$ without the restriction $\inf\limits_{e\in \mathcal{E}\left(
\Gamma \right) }\left\vert e\right\vert =d_{\ast }>0.$
\end{remark}

\begin{remark}
Theorem \ref{persson} admits an obvious extension to the following general
case. Suppose Hypothesis \ref{hyp copy(1)} holds$.$ Define%
\begin{eqnarray}
\mathbf{s}_{q}\left[ f\right] &=&\int\nolimits_{\Gamma }p(x)\left\vert
f^{\prime }(x)\right\vert ^{2}+q\left\vert f(x)\right\vert ^{2}\mathrm{d}x,%
\text{ } \\
\mathbf{s}_{q,n}\left[ f\right] &=&\int\nolimits_{\Gamma \backslash \Gamma
_{n}}p(x)\left\vert f^{\prime }(x)\right\vert ^{2}+q\left\vert
f(x)\right\vert ^{2}\mathrm{d}x
\end{eqnarray}%
on the respective domains
\begin{eqnarray*}
\text{\emph{dom}}(\mathbf{s}_{q}) &=&\{f\in H^{1}(\Gamma ;p,w):\left\vert
\mathbf{q}\left[ f\right] \right\vert <\infty \},\text{ } \\
\text{\emph{dom}}(\mathbf{s}_{q,n}) &=&\{f\in H^{1}(\Gamma \backslash \Gamma
_{n};p,w):\left\vert \mathbf{q}_{n}\left[ f\right] \right\vert <\infty \}.
\end{eqnarray*}%
By a similar proof to that of Lemma \ref{s close}, one has $\mathbf{s}_{q}$
and $\mathbf{s}_{q,n}$ are all semibounded and closed; respectively, denote $%
\mathbf{H}_{\mathbf{s}_{q}}$ and $\mathbf{H}_{\mathbf{s}_{q,n}}$ the
corresponding self-adjoint operators. Then
\begin{equation*}
\inf \sigma _{ess}\left( \mathbf{H}_{\mathbf{s}_{q}}\right)
=\lim_{n\rightarrow \infty }\inf \sigma \left( \mathbf{H}_{\mathbf{s}%
_{q,n}}\right) .
\end{equation*}
\end{remark}

\noindent   Yihan Liu\\
\noindent   School of Mathematics \\
\noindent   Tianjin University \\
\noindent   Tianjin\\
\noindent   300354\\
\noindent   People's Republic of China\\
\noindent   yihanliu@tju.edu.cn\\

\noindent   Jun Yan\\
\noindent   School of Mathematics \\
\noindent   Tianjin University \\
\noindent   Tianjin\\
\noindent   300354\\
\noindent   People's Republic of China\\
\noindent   jun.yan@tju.edu.cn\\

\noindent   Jia Zhao\\
\noindent   Department of Mathematics\\
\noindent   School of Science\\
\noindent   Hebei University of Technology\\
\noindent   Tianjin\\
\noindent   300401\\
\noindent   People's Republic of China\\
\noindent   zhaojia@hebut.edu.cn\\

\end{document}